\documentclass{amsart}
\usepackage{amssymb,amsmath,amsfonts,amsthm,graphicx,enumerate,amscd,latexsym,curves,multibox}
\usepackage{bbm}
\usepackage{mathrsfs}
\usepackage[mathscr]{eucal}
\usepackage{epsfig,epsf,xypic,epic}

\usepackage[T1]{fontenc}
\usepackage[sc]{mathpazo}

\def\bZ{\mathbb{Z}}
\def\bQ{\mathbb{Q}}
\def\bR{\mathbb{R}}
\def\bC{\mathbb{C}}
\def\bP{\mathbb{P}}

\newtheorem{thm}{Theorem}[section]
\newtheorem{lem}{Lemma}[section]

\newtheorem{prop}{Proposition}[section]
\newtheorem{defn}{Definition}[section]

\numberwithin{equation}{section}

\begin{document}

\title[Open and closed GW invariants]{A formula equating open and closed Gromov-Witten invariants and its applications to mirror symmetry}
\author[K. Chan]{Kwokwai Chan}
\address{Department of Mathematics\\ The Chinese University of Hong Kong\\ Shatin \\ Hong Kong}
\email{kwchan@math.cuhk.edu.hk}


\begin{abstract}
We prove that open Gromov-Witten invariants for semi-Fano toric manifolds of the form $X=\bP(K_Y\oplus\mathcal{O}_Y)$, where $Y$ is a toric Fano manifold, are equal to certain 1-pointed closed Gromov-Witten invariants of $X$. As applications, we compute the mirror superpotentials for these manifolds. In particular, this gives a simple proof for the formula of the mirror superpotential for the Hirzebruch surface $\mathbb{F}_2$.
\end{abstract}

\maketitle

\tableofcontents

\section{Introduction}

Let $X$ be a compact complex $n$-dimensional toric manifold equipped with a toric K\"ahler structure $\omega$. Let $L$ be a Lagrangian torus fiber of the moment map associated to $(X,\omega)$. In \cite{FOOO08a}, Fukaya, Oh, Ohta and Ono define open Gromov-Witten invariants for $(X,L)$ as follows. Let $\beta\in\pi_2(X,L)$ be a relative homotopy class with Maslov index $\mu(\beta)=2$. Let $\mathcal{M}_1(L,\beta)$ be the moduli space of holomorphic disks in $X$ with boundaries lying in $L$ and with one boundary marked point representing the class $\beta$. A compactification of $\mathcal{M}_1(L,\beta)$ is given by the moduli space $\overline{\mathcal{M}}_1(L,\beta)$ of stable maps from genus 0 bordered Riemann surfaces $(\Sigma,\partial\Sigma)$ to $(X,L)$ with one boundary marked point representing the class $\beta$. As shown by Fukaya et al. in their monumental work \cite{FOOO06}, $\overline{\mathcal{M}}_1(L,\beta)$ is a Kuranishi space with real virtual dimension $n$. By Corollary 11.5 in \cite{FOOO08a}, there exists a virtual fundamental $n$-cycle $[\overline{\mathcal{M}}_1(L,\beta)]^\textrm{vir}$. The pushforward of this cycle by the evaluation map $ev:\overline{\mathcal{M}}_1(L,\beta)\to L$ at the boundary marked point then gives
$$c_\beta=ev_*([\overline{\mathcal{M}}_1(L,\beta)]^\textrm{vir})\in H_n(L,\bQ)\cong\bQ.$$
By Lemma 11.7 in \cite{FOOO08a}, the homology class $c_\beta$ is independent of the pertubation data (transversal multisections) used to define $[\overline{\mathcal{M}}_1(L,\beta)]^\textrm{vir}$. Hence, $c_\beta$ is an open Gromov-Witten invariant for $(X,L)$.

Let $Y$ be an $(n-1)$-dimensional toric Fano manifold. Consider the $\bP^1$-bundle $X=\bP(K_Y\oplus\mathcal{O}_Y)$ over $Y$, where $K_Y$ denotes the anti-canonical bundle of $Y$. Then $X$ is an $n$-dimensional toric manifold which is semi-Fano, i.e. the anti-canonical bundle $K_X^{-1}$ is nef. Let $h\in H_2(X,\bZ)$ be the fiber class. Let $\alpha\in H_2(X,\bZ)$ be an effective class with $c_1(\alpha)=c_1(X)\cdot\alpha=0$. Consider the moduli space $\overline{\mathcal{M}}_{0,1}(X,h+\alpha)$ of genus 0 stable maps to $X$ with one marked point representing the class $h+\alpha$.\footnote{Here the subscripts $0,1$ denote the genus and number of marked points respectively.} By \cite{FO99}, $\overline{\mathcal{M}}_{0,1}(X,h+\alpha)$ is a Kuranishi space with complex virtual dimension $n$. The pushforward of the virtual fundamental cycle $[\overline{\mathcal{M}}_{0,1}(X,h+\alpha)]^\textrm{vir}$ by the evaluation map $ev:\overline{\mathcal{M}}_{0,1}(X,h+\alpha)\to X$ gives a 1-pointed closed Gromov-Witten invariant of $X$:
$$\textrm{GW}^{X,h+\alpha}_{0,1}(\textrm{[pt]})=ev_*([\overline{\mathcal{M}}_{0,1}(X,h+\alpha)]^\textrm{vir})\in H_{2n}(X,\bQ)\cong\bQ,$$
where $\textrm{[pt]}$ denotes the Poincar\'e dual of a point in $X$.

Now let $\iota_0:Y\hookrightarrow X$ be the closed embedding of $Y$ as the zero section of $K_Y$. The image is a toric prime divisor $D_0=\iota_0(Y)\subset X$. As above, we equip $X$ with a toric K\"ahler structure $\omega$ and fix a Lagrangian torus fiber $L$ in $X$. Corresponding to $D_0$ is a relative homotopy class $\beta_0\in\pi_2(X,L)$. More precisely, $\beta_0\in\pi_2(X,L)$ is the class such that $D_i\cdot\beta_0=\delta_{i0}$ for any toric prime divisor $D_i$ in $X$. The main result of this paper is the following formula.
\begin{thm}\label{main_thm}
For the $\bP^1$-bundle $X=\bP(K_Y\oplus\mathcal{O}_Y)$ over a toric Fano manifold $Y$, we have the equality
$$c_{\beta_0+\alpha}=\textrm{GW}^{X,h+\alpha}_{0,1}(\textrm{[pt]})$$
for any effective class $\alpha\in H_2(X,\bZ)$ with $c_1(\alpha)=0$.
\end{thm}
Note that $\beta_0+\alpha\in\pi_2(X,L)$ is a Maslov index two class since $c_1(\alpha)=0$. We will prove this formula in Section \ref{sec4} by comparing the Kuranishi structures of $\overline{\mathcal{M}}_1(L,\beta_0+\alpha)$ and $\overline{\mathcal{M}}_{0,1}(X,h+\alpha)$.

We can apply this formula to study mirror symmetry. Recall that the mirror of a compact toric $n$-fold $X$ is given by a Landau-Ginzburg model $(X^\vee,W)$ consisting of a bounded domain $X^\vee\subset(\bC^*)^n$ and a holomorphic function $W:X^\vee\to\bC$ called the mirror superpotential. In \cite{FOOO08a} (see also Cho-Oh \cite{CO03}, Auroux \cite{A07,A09}, Chan-Leung \cite{CL08a,CL08b}), Fukaya, Oh, Ohta and Ono show that the mirror superpotential can be expressed as a power series whose coefficients are the open Gromov-Witten invariants defined above. However, when $X$ is non-Fano, these invariants are in general very hard to compute. The only known examples are the mirror superpotentials for the Hirzebruch surfaces $\mathbb{F}_2$ and $\mathbb{F}_3$, first computed by Auroux in \cite{A09} using degeneration methods and wall-crossing formulas. More recently, Fukaya, Oh, Ohta and Ono \cite{FOOO10a} give a different proof for the $\mathbb{F}_2$ case.

As an immediate application of our formula, we can express the mirror superpotential of $X=\bP(K_Y\oplus\mathcal{O}_Y)$ in terms of 1-point closed Gromov-Witten invariants (see Theorem \ref{thm5.1}). In particular, since $\mathbb{F}_2=\bP(K_{\bP^1}\oplus\mathcal{O}_{\bP^1})$ and its Gromov-Witten invariants are easy to compute as it is symplectomorphic to $\bP^1\times\bP^1$, this gives a very simple proof of the formula for the mirror superpotential of $\mathbb{F}_2$. See the example in Section \ref{sec5}. Our formula has since then been applied to study mirror symmetry for various classes of toric manifolds. See Lau-Leung-Wu \cite{LLW10a,LLW10b}, Chan-Lau-Leung \cite{CLL10} and Chan-Lau \cite{CL10} for more details.

The rest of this paper is organized as follows. In Section \ref{sec2}, we briefly review Kuranishi spaces and recall the results that we need in this paper. In Section \ref{sec3}, we establish several preliminary results concerning the toric manifolds $X=\bP(K_Y\oplus\mathcal{O}_Y)$. In Section \ref{sec4} we prove our formula by a direct comparison of Kuranishi structures. In Section \ref{sec5}, we discuss applications of our formula to mirror symmetry.

\section*{Acknowledgment}
The ideas in this paper were motivated by the joint work \cite{CLL10} with Siu-Cheong Lau and Prof. Conan Leung. I would like to thank both of them for numerous very useful and inspiring discussions. Thanks are also due to Baosen Wu for generously sharing many of his ideas and insights. I am also grateful to Prof. Shing-Tung Yau and Mr. Sze-Lim Li for their continuous encouragement and support. I also thank IH\'{E}S for hospitality and providing an excellent research environment during my stay in the fall of 2010 when the revision of this paper was done. This research was partially supported by Harvard University, a Croucher Foundation Fellowship and a William Hodge Fellowship.

\section{Kuranishi structures}\label{sec2}

In this section, we briefly review the theory of Kuranishi spaces and recall some of their properties for later use. We follow Appendix A1 in the book \cite{FOOO06}. See also Section 3 in Fukaya-Ono \cite{FO99}.

Let $\mathcal{M}$ be a compact metrizable space.
\begin{defn}[Definitions A1.1, A1.3, A1.5 in \cite{FOOO06}]
A Kuranishi structure on $\mathcal{M}$ of (real) virtual dimension $d$ consists of the following data:
\begin{enumerate}
\item[(1)] For each point $\sigma\in\mathcal{M}$,
           \begin{enumerate}
           \item[(1.1)] A smooth manifold $V_\sigma$ (with boundary or corners) and a finite group $\Gamma_\sigma$ acting smoothly and effectively on $V_\sigma$.
           \item[(1.2)] A real vector space $E_\sigma$ on which $\Gamma_\sigma$ has a linear representation and such that $\textrm{dim }V_\sigma-\textrm{dim }E_\sigma=d$.
           \item[(1.3)] A $\Gamma_\sigma$-equivariant smooth map $s_\sigma:V_\sigma\to E_\sigma$.
           \item[(1.4)] A homeomorphism $\psi_\sigma$ from $s_\sigma^{-1}(0)/\Gamma_\sigma$ onto a neighborhood of $\sigma$ in $\mathcal{M}$.
           \end{enumerate}
\item[(2)] For each $\sigma\in\mathcal{M}$ and for each $\tau\in\textrm{Im }\psi_\sigma$,
           \begin{enumerate}
           \item[(2.1)] A $\Gamma_\tau$-invariant open subset $V_{\sigma\tau}\subset V_\tau$ containing $\psi_\tau^{-1}(\tau)$.\footnote{Here and in C2 below, we regard $\psi_\tau$ as a map from $s_\tau^{-1}(0)$ to $\mathcal{M}$ by composing with the quotient map $V_\tau\to V_\tau/\Gamma_\tau$.}
           \item[(2.2)] A homomorphism $h_{\sigma\tau}:\Gamma_\tau\to\Gamma_\sigma$.
           \item[(2.3)] An $h_{\sigma\tau}$-equivariant embedding $\varphi_{\sigma\tau}:V_{\sigma\tau}\to V_\sigma$ and an injective $h_{\sigma\tau}$-equivariant bundle map $\hat\varphi_{\sigma\tau}:E_\tau\times V_{\sigma\tau}\to E_\sigma\times V_\sigma$ covering $\varphi_{\sigma\tau}$.
           \end{enumerate}
\end{enumerate}
Moreover, these data should satisfy the following conditions:
\begin{enumerate}
\item[(i)] $\hat\varphi_{\sigma\tau}\circ s_\tau=s_\sigma\circ\varphi_{\sigma\tau}$.\footnote{Here and after, we also regard $s_\sigma$ as a section $s_\sigma:V_\sigma\to E_\sigma\times V_\sigma$.}
\item[(ii)] $\psi_\tau=\psi_\sigma\circ\varphi_{\sigma\tau}$.
\item[(iii)] If $\xi\in\psi_\tau(s_\tau^{-1}(0)\cap V_{\sigma\tau}/\Gamma_\tau)$, then in a sufficiently small neighborhood of $\xi$,
    $$\varphi_{\sigma\tau}\circ\varphi_{\tau\xi}=\varphi_{\sigma\xi},\ \hat\varphi_{\sigma\tau}\circ\hat\varphi_{\tau\xi}=\hat\varphi_{\sigma\xi}.$$
\end{enumerate}
\end{defn}
The spaces $E_\sigma$ are called obstruction spaces (or obstruction bundles), the maps $\{s_\sigma:V_\sigma\to E_\sigma\}$ are called Kuranishi maps, and $(V_\sigma,E_\sigma,\Gamma_\sigma,s_\sigma,\psi_\sigma)$ is called a Kuranishi neighborhood of $\sigma\in\mathcal{M}$.

To define virtual fundamental chains, we need Kuranishi spaces with extra structures.
\begin{defn}[Definitions A1.14, A1.17 in \cite{FOOO06}]
A Kuranishi space is said to have a tangent bundle if the differential of $s_\sigma$ in the direction of the normal bundle induces a bundle isomorphism
\begin{equation}\label{tangent_bundle}
ds_\sigma:\frac{\varphi_{\sigma\tau}^*TV_\sigma}{TV_{\sigma\tau}}\cong
\frac{E_\sigma\times V_{\sigma\tau}}{\hat\varphi_{\sigma\tau}(E_\tau\times V_{\sigma\tau})}
\end{equation}
as $\Gamma_\tau$-equivariant bundles on $V_{\sigma\tau}$.

For a Kuranishi space with tangent bundle, an orientation consists of trivializations of $\Lambda^\textrm{top}E_\sigma^*\otimes\Lambda^\textrm{top}TV_\sigma$ compatible with the isomorphisms (\ref{tangent_bundle}).
\end{defn}
We will not give the precise definition of multisections here. See Definitions A1.19, A1.21 in \cite{FOOO06} for details. Roughly speaking, a multisection $\mathfrak{s}$ is a system of multi-valued perturbations $\{s_\sigma':V_\sigma\to E_\sigma\}$ of the Kuranishi maps $\{s_\sigma:V_\sigma\to E_\sigma\}$ satisfying certain compatibility conditions. For a Kuranishi space $\mathcal{M}$ with tangent bundle, there exist (a family of) multisections $\mathfrak{s}$ which are transversal to 0 (Theorem A1.23 in \cite{FOOO06}). Furthermore, suppose that $\mathcal{M}$ is oriented. Let $ev:\mathcal{M}\to Y$ be a strongly smooth map to a smooth manifold $Y$, i.e. a family of $\Gamma_\sigma$-invariant smooth maps $\{ev_\sigma:V_\sigma\to Y\}$ such that $ev_\sigma\circ\varphi_{\sigma\tau}=ev_\tau$ on $V_{\sigma\tau}$. Then, using these transversal multisections, one can define the virtual fundamental chain $ev_*([\mathcal{M}]^{\textrm{vir}})$ as a $\bQ$-singular chain in $Y$ (Definition A1.28 in \cite{FOOO06}).

We will also need the notion of fiber products of Kuranishi spaces. See Appendix A1.2 in \cite{FOOO06} for more details. As before, let $ev:\mathcal{M}\to Y$ be a strongly smooth map from a Kuranishi space $\mathcal{M}$ to a smooth manifold $Y$. Suppose that $ev$ is weakly submersive, i.e. each $ev_\sigma:V_\sigma\to Y$ is a submersion. Let $W$ be another manifold and $g:W\to Y$ be a smooth map. Consider the fiber product
$$\mathcal{Z}=\mathcal{M}\times_Y W=\{(\sigma,p)\in\mathcal{M}\times W:ev(\sigma)=q(p)\}.$$
\begin{defn}[Definition A1.37 in \cite{FOOO06}]\label{fiber_product}
Let $(\sigma,p)\in\mathcal{Z}$ and $(V_\sigma,E_\sigma,\Gamma_\sigma,s_\sigma,\psi_\sigma)$ be a Kuranishi neighborhood of $\sigma\in\mathcal{M}$. We set
$$V_{(\sigma,p)}=\{(\tau,q)\in V_\sigma\times W:ev_\sigma(\tau)=g(q)\}.$$
Then $V_{(\sigma,p)}$ is a smooth manifold since $ev_\sigma$ is a submersion. We also set $E_{(\sigma,p)}=E_\sigma, \Gamma_{(\sigma,p)}=\Gamma_\sigma$ and define $s_{(\sigma,p)},\psi_{(\sigma,p)}$ in the obvious way. This defines a Kuranishi neighborhood of $(\sigma,p)\in\mathcal{Z}$, and they glue together to give a Kuranishi structure on $\mathcal{Z}$.
\end{defn}
\begin{lem}[Lemma A1.39 in \cite{FOOO06}]\label{lem2.1}
If the Kuranishi space $\mathcal{M}$ has a tangent bundle, so does the Kuranishi structure on $\mathcal{Z}$. Furthermore, if the Kuranishi structure on $\mathcal{M}$ and the manifolds $Y$, $W$ are all oriented, so is the Kuranishi structure on $\mathcal{Z}$.
\end{lem}
Let $\hat{ev}:\mathcal{Z}\to W$ be the projection map. We remark that this is a strongly smooth map. The following lemma is crucial to the proof of our main result.
\begin{lem}[Lemma A1.43 in \cite{FOOO06}]\label{lem2.2}
Suppose that $Y$ and $W$ are oriented and compact without boundary, and $\partial\mathcal{M}=\emptyset$. Then we have
$$\textrm{PD}(\hat{ev}_*([\mathcal{Z}]^\textrm{vir}))=g^*(\textrm{PD}(ev_*([\mathcal{M}]^\textrm{vir}))),$$
where $\textrm{PD}$ denotes Poincar\'e dual.
\end{lem}

\section{A class of semi-Fano toric manifolds}\label{sec3}

Let $Y$ be an $(n-1)$-dimensional toric Fano manifold. Denote by $K_Y$ its canonical line bundle. Consider the $\bP^1$-bundle $X=\bP(K_Y\oplus\mathcal{O}_Y)$ over $Y$. In this section, we shall establish some elementary properties of the toric manifold $X$ which will be of use later.

Let $e_1,\ldots,e_n$ be the standard basis of a rank $n$ lattice $N\cong\bZ^n$, and let $N'=\{v=\sum_{j=1}^nv^je_j\in N|v^n=0\}\cong\bZ^{n-1}$. Let $N_\bR=N\otimes_\bZ\bR$, $N'_\bR=N'\otimes_\bZ\bR$. Without loss of generality, we can choose the primitive generators of the 1-dimensional cones of the fan $\Delta$ in $N_\bR$ defining $X$ to be
$$v_0=e_n,v_1=w_1+e_n,\ldots,v_m=w_m+e_n,v_{m+1}=-e_n,$$
where $w_1,\ldots,w_m\in N'$ are the primitive generators of the 1-dimensional cones of a fan $\Delta'$ in $N'_\bR$ defining $Y$.
\begin{lem}\label{lem3.1}
Let $h\in H_2(X,\bZ)$ be the fiber class of the $\bP^1$-bundle $X=\bP(K_Y\oplus\mathcal{O}_Y)$. Let $\iota_0:Y\hookrightarrow X=\bP(K_Y\oplus\mathcal{O}_Y)$ be the closed embedding which maps $Y$ to the zero section of $K_Y$. $\iota_0$ induces an embedding $\iota_{0*}:H_2(Y,\bZ)\hookrightarrow H_2(X,\bZ)$. Then we have
$$H_2^\textrm{eff}(X,\bZ)\cong\bZ_{\geq0}h\oplus\iota_{0*}H_2^\textrm{eff}(Y,\bZ).$$
Here, the superscript "eff" refers to effective classes. Moreover, we have $c_1(h)=2$ and $c_1(\alpha)=0$ for any $\alpha\in\iota_{0*}H_2^\textrm{eff}(Y,\bZ)$. In particular, $X$ is semi-Fano, i.e. the anti-canonical bundle $K_X^{-1}$ is nef.
\end{lem}
\begin{proof}
Recall that a subset $\mathscr{P}=\{v_{i_1},\ldots,v_{i_p}\}\subset\{v_0,\ldots,v_{m+1}\}$ is called a \textit{primitive collection} if for each $1\leq k\leq p$, the elements of $\mathscr{P}\setminus\{x_{i_k}\}$ generate a $(p-1)$-dimensional cone in $\Delta$ but $\mathscr{P}$ itself does not generate a cone in $\Delta$ (see Batyrev \cite{B91}). The \textit{focus} $\sigma(\mathscr{P})$ of $\mathscr{P}$ is the cone in $\Delta$ of the smallest dimension which contains $v_{i_1}+\ldots+v_{i_p}$. Let $v_{j_1},\ldots,v_{j_q}$ be the generators of $\sigma(\mathscr{P})$. Then there exists positive integers $n_1,\ldots,n_q$ such that
$$v_{i_1}+\ldots+v_{i_p}=n_1v_{j_1}+\ldots+n_qv_{j_q}.$$
This is known as a \textit{primitive relation}. Recall that the homology group $H_2(X,\bZ)$ is given by the kernel of the surjective map $\bZ^d\rightarrow N,\ E_i\mapsto v_i$, where $\{E_1,\ldots,E_d\}$ is the standard basis of $\bZ^d$. Also, the effective cone $H_2^\textrm{eff}(X,\bZ)$ is generated by primitive relations.

In our case, $\mathscr{P}_0:=\{v_0,v_{m+1}\}$ is obviously a primitive collection for $\Delta$. The primitive relation $v_0+v_{m+1}=0$ corresponds to the fiber class $h$ of the $\bP^1$-bundle $X\to Y$. It is obvious that we have $c_1(h)=c_1(X)\cdot h=2$.

By Proposition 4.1 in Batyrev \cite{B91}, we have $\mathscr{P}\cap\mathscr{P}_0=\emptyset$ for any other primitive collection $\mathscr{P}\neq\mathscr{P}_0$. Suppose that $\mathscr{P}\neq\mathscr{P}_0$ is a primitive collection consisting of the elements $v_{i_1}=w_{i_1}+e_n,\ldots,v_{i_p}=w_{i_p}+e_n$, where $1\leq i_1<\ldots<i_p\leq m$. Then $\mathscr{P}':=\{w_{i_1},\ldots,w_{i_p}\}$ is obviously a primitive collection for the fan $\Delta'$ defining $Y$. Now, let $w_{j_1},\ldots,w_{j_q}$ be the generators of the focus $\sigma(\mathscr{P}')$ of $\mathscr{P}'$. The primitive relation for $\Delta'$ is given by
\begin{equation}\label{P'}
w_{i_1}+\ldots+w_{i_p}=n_1w_{j_1}+\ldots+n_qw_{j_q},
\end{equation}
for some $n_1,\ldots,n_q\in\bZ_{>0}$. Let $\gamma\in H_2^\textrm{eff}(Y,\bZ)$ be the corresponding effective class. Since $Y$ is Fano, we have $p-n_1-\ldots-n_q=c_1(Y)\cdot\gamma>0$. In terms of the $v_i$'s, (\ref{P'}) becomes a primitive relation
$$v_{i_1}+\ldots+v_{i_p}=n_1v_{j_1}+\ldots+n_qv_{j_q}+(p-n_1-\ldots-n_q)v_0$$
for $\Delta$. This corresponds to the class $\alpha:=\iota_{0*}(\gamma)\in H_2^\textrm{eff}(X,\bZ)$, whose Chern number is given by $c_1(\alpha)=p-n_1-\ldots-n_q-(p-n_1-\ldots-n_q)=0$.
\end{proof}
As usual, denote by $D_0,D_1,\ldots,D_m,D_{m+1}$ the toric prime divisors corresponding to the primitive generators $v_0,v_1,\ldots,v_m,v_{m+1}$ respectively. Note that $D_0=\iota_0(Y)$.
\begin{lem}\label{lem3.2}
Let $\varphi:\bP^1\rightarrow X$ be a nonconstant holomorphic map from $\bP^1$ to $X$.
\begin{enumerate}
\item[(1)] Suppose that $[\varphi(\bP^1)]=h+\alpha\in H_2(X,\bZ)$ for some $\alpha\in\iota_{0*}H_2^\textrm{eff}(Y,\bZ)$. Then $\varphi(\bP^1)$ is contained in one of the toric prime divisors $D_0,D_1,\ldots,D_m$.
\item[(2)] Suppose that $[\varphi(\bP^1)]=\alpha\in\iota_{0*}H_2^\textrm{eff}(Y,\bZ)$. Then $\varphi(\bP^1)$ is contained in the toric prime divisor $D_0$.
\end{enumerate}
\end{lem}
\begin{proof}
Suppose that $\varphi:\bP^1\rightarrow X$ is a nonconstant holomorphic map with class $h+\alpha$ for some $\alpha\in\iota_{0*}H_2^\textrm{eff}(Y,\bZ)$. From the proof of the above lemma, we know that the class $h+\alpha$ corresponds to the primitive relation
$$(1-\sum_{i=1}^ma_i)v_0+\sum_{i=1}^ma_iv_i+v_{m+1}=0.$$
Moreover, we have $\sum_{i=1}^ma_i\geq1$, and if $\sum_{i=1}^ma_i=1$, then there exists $1\leq i\leq m$ such that $a_i<0$. Hence there exists $0\leq i\leq m$ such that $D_i\cdot\varphi(\bP^1)=D_i\cdot(h+\alpha)<0$. This implies that $\varphi(\bP^1)$ is contained in $D_i$. This proves (1). (2) can be proved in the same way.
\end{proof}

\section{Proof of Theorem \ref{main_thm}}\label{sec4}

We can now start our proof of Theorem \ref{main_thm}.

We equip $X=\bP(K_Y\oplus\mathcal{O}_Y)$ with a toric K\"ahler structure $\omega$. Let $L\subset X$ be a Lagrangian torus fiber of the associated moment map. For $i=0,1,\ldots,m,m+1$, let $\beta_i\in\pi_2(X,L)$ be the relative homotopy class such that $D_j\cdot\beta_i=\delta_{ij}$. Then each $\beta_i$ is a Maslov index two class with $\partial\beta_i=v_i$, where $\partial:\pi_2(X,L)\to\pi_1(L)$ is the boundary map, and $\pi_2(X,L)$ is generated by $\beta_0,\beta_1,\ldots,\beta_m,\beta_{m+1}$. Moreover, each $\beta_i$ is represented by a family of holomorphic disks $\varphi_i:(D^2,\partial D^2)\to(X,L)$. Here, $D^2=\{z\in\bC:|z|\leq1\}$ is the unit disk.

Fix a nonzero effective class $\alpha\in H_2^\textrm{eff}(X,\bZ)$ with $c_1(\alpha)=0$. Let $\overline{\mathcal{M}}_1(L,\beta_0+\alpha)$ be the moduli space of stable maps from genus 0 bordered Riemann surfaces to $(X,L)$ with one boundary marked point representing the class $\beta_0+\alpha$. To simplify notations, we denote $\overline{\mathcal{M}}_1(L,\beta_0+\alpha)$ by $\mathcal{M}^L$. Similarly, we denote by $\mathcal{M}^X$ the moduli space $\overline{\mathcal{M}}_{0,1}(X,h+\alpha)$ of genus 0 stable maps to $X$ with one marked point representing the class $h+\alpha$. We have evaluation maps\footnote{By a slight abuse of notations, we use $ev$ to denote both evaluation maps. It should clear from the context which one we are referring to.}
\begin{eqnarray*}
ev:\mathcal{M}^L\to L,\ ev:\mathcal{M}^X\to X.
\end{eqnarray*}

By \cite{FOOO06}, both $\mathcal{M}^L$ and $\mathcal{M}^X$ are oriented Kuranishi spaces with tangent bundles, and the evaluation maps are both strongly smooth and weakly submersive. The real virtual dimensions of $\mathcal{M}^L$ and $\mathcal{M}^X$ are $n$ and $2n$ respectively. Moreover, since $\mu(\beta_0+\alpha)=2$, we have $\partial\mathcal{M}^L=\emptyset$ by Corollary 11.5 in \cite{FOOO08a}. It is also well-known that $\mathcal{M}^X$ has no boundary. Hence, they define virtual fundamental cycles
$$ev_*([\mathcal{M}^L]^\textrm{vir})\in H_n(L,\bQ),\ ev_*([\mathcal{M}^X]^\textrm{vir})\in H_{2n}(X,\bQ).$$

Fix a point $p\in L\subset X$. Let $\iota:\{p\}\hookrightarrow L$ (resp. $\iota:\{p\}\hookrightarrow X$) be the inclusion of the point $p$. We can then apply the construction in Definition \ref{fiber_product} and Lemma \ref{lem2.1} to give oriented Kuranishi structures with tangent bundles on the spaces:
$$\mathcal{M}^L_p:=\mathcal{M}^L\times_L\{p\},\ \mathcal{M}^X_p:=\mathcal{M}^X\times_X\{p\}.$$
Both have real virtual dimension 0. Let $\hat{ev}:\mathcal{M}^L_p\to\{p\}$, $\hat{ev}:\mathcal{M}^X_p\to\{p\}$ be the induced (constant) maps. Then we have virtual fundamental cycles
$$\hat{ev}_*([\mathcal{M}^L_p]^\textrm{vir}),\hat{ev}_*([\mathcal{M}^X_p]^\textrm{vir})\in H_0(\{p\},\bQ)\cong\bQ.$$

Now Lemma \ref{lem2.2} says that
\begin{prop}\label{prop4.1} We have
\begin{eqnarray*}
\textrm{PD}(\hat{ev}_*([\mathcal{M}^L_p]^\textrm{vir})) & = & \iota^*\textrm{PD}(ev_*([\mathcal{M}^L]^\textrm{vir}))\\
\textrm{PD}(\hat{ev}_*([\mathcal{M}^X_p]^\textrm{vir})) & = & \iota^*\textrm{PD}(ev_*([\mathcal{M}^X]^\textrm{vir}))
\end{eqnarray*}
in $H^0(\{p\},\bQ)\cong\bQ$.
\end{prop}

Therefore, to prove Theorem \ref{main_thm}, it suffices to show that $\mathcal{M}^L_p$ and $\mathcal{M}^X_p$ have the same Kuranishi structures.

To do this, we first show that $\mathcal{M}^L_p$ can naturally be identified with $\mathcal{M}^X_p$ as a set. Let us recall the following results proved by Cho and Oh in \cite{CO03}, which holds for general toric manifolds.
\begin{thm}[Theorem 5.2 in \cite{CO03}; see also Theorem 11.1 in \cite{FOOO08a}]\label{Cho-Oh} Let $(X,\omega)$ be a toric K\"ahler manifold and $L$ be a Lagrangian torus fiber of its moment map. Let $D_1,\ldots,D_d$ be all the toric prime divisors in $X$ and $\beta_1,\ldots,\beta_d\in\pi_2(X,L)$ be the relative homotopy classes such that $D_j\cdot\beta_i=\delta_{ij}$.
\begin{enumerate}
\item[(1)] If $\varphi:(D^2,\partial D^2)\to(X,L)$ is a holomorphic map from a disk representing a Maslov index two class $\beta\in\pi_2(X,L)$, then $\beta=\beta_i$ for some $i\in\{1,\ldots,d\}$.
\item[(2)] For $i=1,\ldots,d$, let $\overline{\mathcal{M}}_1(L,\beta_i)$ be the moduli space of stable maps from genus 0 bordered Riemann surfaces to $(X,L)$ with one boundary marked point representing the class $\beta_i$. Then the evaluation map $ev:\overline{\mathcal{M}}_1(L,\beta_i)\to L$ is an orientation preserving diffeomorphism. In particular, for any $p\in L$ and any $i\in\{1,\ldots,d\}$, there is a unique (up to automorphisms of the domain) genus 0 bordered stable map whose boundary passes through $p$ and whose domain is a disk which represents the class $\beta_i$.
\end{enumerate}
\end{thm}
Now, let $\sigma^L=((\Sigma^L,z),\varphi)$ be representing a point in $\mathcal{M}^L_p$. This consists of a genus 0 bordered Riemann surface $\Sigma^L$ with a boundary marked point $z\in\partial\Sigma^L$ and a stable map $\varphi:(\Sigma^L,\partial\Sigma^L)\to(X,L)$ such that $\varphi(z)=p$.
\begin{prop}\label{prop4.2}
$\Sigma^L$ can be decomposed as $\Sigma^L=\Sigma^L_0\cup\Sigma_1$, where $\Sigma^L_0=D^2$ is a disk and $\Sigma_1$ is a genus 0 nodal curve, such that the restrictions $\varphi_0:=\varphi|_{\Sigma^L_0}$ and $\varphi_1:=\varphi|_{\Sigma_1}$ represent the classes $\beta_0$ and $\alpha$ respectively.
\end{prop}
\begin{proof}
The Maslov index of $\beta_0+\alpha$ is $\mu(\beta_0+\alpha)=2$ since $c_1(\alpha)=0$. By Theorem \ref{Cho-Oh}(1), there does not exist any nonconstant holomorphic map from a disk to $(X,L)$ with class $\beta_0+\alpha$, so $\Sigma^L$ must be singular. Decompose $\Sigma^L$ into irreducible components. Let $\varphi_j:(D^2,\partial D^2)\to(X,L)$ and $\varphi_k:\bP^1\to X$ be the restriction of $\varphi$ to the disk and sphere components respectively. Then $\beta_0+\alpha=\sum_j[\varphi_j]+\sum_k[\varphi_k]$. Notice that, by the proof of Lemma \ref{lem3.1}, any $\alpha\in H_2(X,\bZ)$ with $c_1(\alpha)=0$ cannot be expressed as a $\bZ$-linear combination of $\beta_i$'s with positive coefficients. Hence, there must be only one disk component in $\Sigma$. Therefore, we can decompose $\Sigma$ into $\Sigma^L_0\cup\Sigma_1$, where $\Sigma^L_0=D^2$ is a disk and $\Sigma_1$ is a genus 0 nodal curve (i.e. a tree of $\bP^1$'s). Now, the restriction $\varphi_0:=\varphi|_{\Sigma^L_0}$ is a nonconstant holomorphic map from $(D^2,\partial D^2)$ to $(X,L)$. By Theorem \ref{Cho-Oh}(1) again, the class of $\varphi_0$ must be $\beta_0$. Hence $\varphi_1:=\varphi|_{\Sigma_1}$ represents $\alpha$.
\end{proof}
\begin{prop}\label{prop4.3}
There exists a unique holomorphic map $\varphi_{m+1}:(D^2,\partial D^2)\to(X,L)$ representing the class $\beta_{m+1}$ such that its boundary $\partial\varphi_{m+1}:=\varphi_{m+1}|_{\partial D^2}$ is exactly given by $\partial\varphi_0:=\varphi_0|_{\partial D^2}$ with the opposite orientation, where $\varphi_0$ is the map obtained in Proposition \ref{prop4.2}
\end{prop}
\begin{proof}
Let $\varphi_{m+1}:(D^2,\partial D^2)\to(X,L)$ be a holomorphic map representing the class $\beta_{m+1}$ such that $p\in\varphi(\partial D^2)$. By Theorem \ref{Cho-Oh}(2), there exists one and only one such map up to automorphisms of $D^2$. Consider the moduli space $\overline{\mathcal{M}}_{0,1}(X,h)$ of genus 0 stable maps to $X$ with one marked point which represent the fiber class $h$. Since $X\to Y$ is a $\bP^1$-bundle, the evaluation map $ev:\overline{\mathcal{M}}_{0,1}(X,h)\to X$ is an isomorphism. Hence, there exists a unique (up to automorphisms of the domain) holomorphic map $\phi:\bP^1\to X$ representing the class $h$ which passes through $p\in L\subset X$. The image of this map is the fiber $C_p\cong\bP^1$ of $X\to Y$ which contains $p$. Now, the intersection $C_p\cap L\cong S^1$ splits the fiber $C_p$ into two disks. This gives two holomorphic maps $\varphi_0':(D^2,\partial D^2)\to(X,L)$ and $\varphi_{m+1}':(D^2,\partial D^2)\to(X,L)$ with classes $\beta_0$ and $\beta_{m+1}$ respectively. By Theorem \ref{Cho-Oh}(2), they must be the same as $\varphi_0$, $\varphi_{m+1}$ up to automorphisms of $D^2$. Hence, by composing $\varphi_{m+1}$ with an automorphism of $D^2$, which is uniquely determined by $\varphi_0$, we get the desired unique holomorphic map representing the class $\beta_{m+1}$.
\end{proof}
By Proposition \ref{prop4.3}, we can glue the maps $\varphi:(\Sigma^L,\partial\Sigma^L)\rightarrow(X,L)$ and $\varphi_{m+1}:(D^2,\partial D^2)\to(X,L)$ together to give a holomorphic map $\varphi':\Sigma\to X$ which represents the class $\beta_0+\beta_{m+1}+\alpha=h+\alpha$, where $\Sigma$ is the union of $\Sigma^L$ and $D^2$ with their boundaries identified in the obvious way. It is easy to see that this map is stable. Hence, $\sigma^X:=((\Sigma,z),\varphi')$ represents a point in $\mathcal{M}^X=\overline{\mathcal{M}}_{0,1}(X,h+\alpha)$ and we have $ev(\sigma)=p$. This defines a map
$$j:\mathcal{M}^L_p\to\mathcal{M}^X_p,\ [\sigma^L]\mapsto[\sigma^X].$$
$j$ is well-defined: Any automorphism of $\sigma^L=((\Sigma^L,z),\varphi)$ acts trivially on the component $\Sigma^L_0$ because $\varphi$ is nonconstant on this component. So any representative of $[\sigma^L]$ is mapped to the same isomorphism class in $\mathcal{M}^X_p$. We need to show that $j$ is bijective.

Let $\sigma^X=((\Sigma,z),\varphi)$ be representing a point in $\mathcal{M}^X_p$. This consists of a genus 0 nodal curve $\Sigma$ with a marked point $z\in\Sigma$ and a stable map $\varphi:\Sigma\to X$ representing the class $h+\alpha$ such that $\varphi(z)=p$. The following is an analog of Proposition \ref{prop4.2}.
\begin{prop}\label{prop4.4}
$\Sigma$ can be decomposed as $\Sigma=\Sigma_0\cup\Sigma_1$, where $\Sigma_0\cong\bP^1$ is irreducible, such that the restrictions $\varphi_0:=\varphi|_{\Sigma_0}$ and $\varphi_1:=\varphi|_{\Sigma_1}$ represent the classes $h$ and $\alpha$ respectively.
\end{prop}
\begin{proof}
By Lemma \ref{lem3.2}(1), there does not exist any nonconstant holomorphic map from $\bP^1$ to $X$ representing the class $h+\alpha$ whose image is not contained entirely in the toric divisors. Hence, $\Sigma$ must be singular. Decompose $\Sigma$ into components $\Sigma=\bigcup_a\Sigma_a$, where each $\Sigma_a\cong\bP^1$ is irreducible. Then we have
$$\sum_a[\varphi(\Sigma_a)]=h+\alpha.$$
Since $h$ is primitive, there exists $a_0$ such that $\varphi(\Sigma_{a_0})=h+\alpha'$ and $\sum_{a\neq a_0}[\varphi(\Sigma_a)]=\alpha"$ for some $\alpha',\alpha"\in\iota_{0*}H_2^\textrm{eff}(Y,\bZ)\subset H_2^\textrm{eff}(X,\bZ)$ with $\alpha=\alpha'+\alpha"$. By Lemma \ref{lem3.1}, we have $c_1(\alpha')=c_1(\alpha")=0$. Then, by Lemma \ref{lem3.2}(2), the images of $\bigcup_{a\neq a_0}\Sigma_a$ is contained entirely in the zero section $D_0$. So the image of $\Sigma_{a_0}$ must be intersecting with $L$ at $p$. Applying Lemma \ref{lem3.2}(1) again, we see that $\alpha'$ must be zero. The result follows.
\end{proof}
Note that $\varphi_0$ is a nonconstant holomorphic map from $\bP^1$ to $X$ whose image contains $p$. Arguing as in the proof of Proposition \ref{prop4.3}, we see that the image of $\varphi_0$ is the fiber $C_p$ of the $\bP^1$-bundle $X\to Y$ which contains $p$, and $\varphi_0(\bP^1)\cap L=S^1$. We can then split $\Sigma_0\cong\bP^1$ into two disks $\Sigma_0=\Sigma_0'\cup\Sigma_0"\cong D^2\cup D^2$, and split $\varphi_0$ into two holomorphic maps $\varphi_0':(\Sigma_0',\partial\Sigma_0')\to(X,L)$ and $\varphi_{m+1}':(\Sigma_0",\partial\Sigma_0")\to(X,L)$ which represent the classes $\beta_0$ and $\beta_{m+1}$ respectively. Now, let $\Sigma^L:=\Sigma_0'\cup\Sigma_1$ and $\varphi':=\varphi|_{\Sigma^L}$. Then $\varphi':(\Sigma^L,\partial\Sigma^L)\to(X,L)$ is a genus 0 bordered stable map such that $\varphi(\partial\Sigma^L)$ contains $p$, and $\sigma^L:=((\Sigma^L,z),\varphi')$ represents a point in $\mathcal{M}^L_p$. By our constructions, $j([\sigma^L])=[\sigma^X]$. This defines a map $j^{-1}:\mathcal{M}^X_p\to\mathcal{M}^L_p$. Again, since any automorphism of $\sigma^X=((\Sigma,z),\varphi)$ acts trivially on the component $\Sigma_0$, the map $j^{-1}$ is well-defined. It is obvious that this is the inverse map of $j$. Hence, $j$ is a bijective map.
\begin{prop}\label{prop4.5}
Under the bijective map $j:\mathcal{M}^L_p\to\mathcal{M}^X_p$, the Kuranishi structures on $\mathcal{M}^L_p$ and $\mathcal{M}^X_p$ can be naturally identified.
\end{prop}
\begin{proof}
We shall first briefly recall the constructions of Kuranishi neighborhoods from \cite{FO99} and \cite{FOOO06}.

We begin with $\mathcal{M}^L_p$. Let $\sigma^L=((\Sigma^L,z),\varphi)$ be representing a point in $\mathcal{M}^L_p$. By Proposition \ref{prop4.2}, we can decompose $\Sigma^L$ into irreducible components $\Sigma^L=\Sigma_0\cup\Sigma_1\cup\ldots\cup\Sigma_k$, where $\Sigma_0=D^2$ is a disk and $\Sigma_1,\ldots,\Sigma_k$ are copies of $\bP^1$, such that the restrictions of $\varphi$ to $\Sigma_0$ and $\bigcup_{a=1}^k\Sigma_a$ represent the classes $\beta_0$ and $\alpha$ respectively.

For each $a=0,1,\ldots,k$, let $W^{1,p}(\Sigma_a;\varphi^*(TX);L)$ be the space of sections $v$ of $\varphi^*(TX)$ of $W^{1,p}$ class such that the restriction of $v$ to $\partial\Sigma_a$ lies in $\varphi^*(TL)$, and $W^{0,p}(\Sigma_a;\varphi^*(TX)\otimes\Lambda^{0,1})$ be the space of sections of $\varphi^*(TX)\otimes\Lambda^{0,1}$ of $W^{0,p}$ class. Note that $L$ does not play a role in the definition of $W^{1,p}(\Sigma_a;\varphi^*(TX);L)$ for $a=1,\ldots,k$. Then, let $W^{1,p}(\Sigma^L;\varphi^*(TX);L)$ be the subspace of $\bigoplus_{a=0}^kW^{1,p}(\Sigma_a;\varphi^*(TX);L)$ consisting of elements $\{u=(u_a)\in\bigoplus_{a=0}^kW^{1,p}(\Sigma_a;\varphi^*(TX);L)$ such that for any singular point $w\in\Sigma^L$ which is the intersection of two irreducible components $\Sigma_a$ and $\Sigma_b$, we have $u_a(w)=u_b(w)$. Also let $W^{0,p}(\Sigma^L;\varphi^*(TX)\otimes\Lambda^{0,1})=\bigoplus_{a=0}^k W^{0,p}(\Sigma_a;\varphi^*(TX)\otimes\Lambda^{0,1})$. Consider the linearization of the Cauchy-Riemann operator $\bar\partial$:
$$D_\varphi\bar\partial:W^{1,p}(\Sigma^L;\varphi^*(TX);L)\to W^{0,p}(\Sigma^L;\varphi^*(TX)\otimes\Lambda^{0,1}).$$
This is a Fredholm operator by ellipticity.

To construct the obstruction space, choose open subsets $W_a$ of $\Sigma_a$ whose closure is disjoint from the boundary of each of $\Sigma_a$ and from the singular and marked points. Then, for each $a=0,1,\ldots,k$, by the unique continuation theorem, we can choose a finite dimensional subset $E_a$ of $C^\infty_0(W_a;\varphi^*(TX))$ such that
$$\textrm{Im }D_\varphi\bar\partial+\bigoplus_{a=0}^kE_a=W^{0,p}(\Sigma^L;\varphi^*(TX)\otimes\Lambda^{0,1}).$$
We also choose $\bigoplus_{a=0}^kE_a$ to be invariant under the group $\Gamma_{\sigma^L}$ of automorphisms of $\sigma^L$. We set $E_{\sigma^L}=\bigoplus_{a=0}^kE_a$.

Let $\Pi:W^{0,p}(\Sigma^L;\varphi^*(TX)\otimes\Lambda^{0,1})\to W^{0,p}(\Sigma^L;\varphi^*(TX)\otimes\Lambda^{0,1})/E_{\sigma^L}$ be the projection map. Let $V_{\textrm{map},\sigma^L}$ be the kernel of the operator $\Pi\circ D_\varphi\bar\partial$. Now, consider the automorphism group $\textrm{Aut}(\Sigma^L,z)$ of the marked bordered Riemann surface $(\Sigma^L,z)$. The group $\textrm{Aut}(\Sigma^L,z)$ may not be finite since some components may be unstable. However, we can naturally embed the Lie algebra $\textrm{Lie(Aut}(\Sigma^L,z))$ into $V_{\textrm{map},\sigma^L}$. Take its $L^2$ orthogonal complement (with respect to a certain metric). Then let $V_{\textrm{map},\sigma^L}'$ be a small neighborhood of the zero of it.

On the other hand, let $V_{\textrm{deform},\sigma^L}$ be a small neighborhood of the origin in the space of first order deformations of the stable components of $(\Sigma^L,z)$. Also let $V_{\textrm{resolve},\sigma^L}$ be a small neighborhood of the origin in the space $\bigoplus_{w}T_w\Sigma_a\otimes T_w\Sigma_b$, where the sum is over singular points $w\in\Sigma^L\setminus\Sigma_0$ and $\Sigma_a$, $\Sigma_b$ are the two components such that $\Sigma_a\cap\Sigma_b=\{w\}$. There is a family of marked semi-stable bordered Riemann surfaces $\{(\Sigma^L_\zeta,z):\zeta\in V_{\textrm{deform},\sigma^L}\times V_{\textrm{resolve},\sigma^L}\}$ over the product $V_{\textrm{deform},\sigma^L}\times V_{\textrm{resolve},\sigma^L}$. We remark that, since we do not deform the singular point in $\Sigma_0$, each $\Sigma^L_\zeta$ is singular and can be decomposed as $\Sigma^L_\zeta=\Sigma_0\cup\Sigma_\zeta'$.

Let $V_{\sigma^L}'=V_{\textrm{map},\sigma^L}'\times V_{\textrm{deform},\sigma^L}\times V_{\textrm{resolve},\sigma^L}$. By the proof of Proposition 12.23 in \cite{FO99}, there exist a $\Gamma_{\sigma^L}$-equivariant smooth map $s_{\sigma^L}:V_{\sigma^L}'\to E_{\sigma^L}$ and a family of smooth maps $\varphi_{u,\zeta}:(\Sigma^L_\zeta,\partial\Sigma^L_\zeta)\to(X,L)$ for $(u,\zeta)\in V_{\sigma^L}'$ such that $\bar\partial\varphi_{u,\zeta}=s_{\sigma^L}(u,\zeta)$. Now we set $V_{\sigma^L}=\{(u,\zeta)\in V_{\sigma^L}':\varphi_{u,\zeta}(z)=p\}$. By abuse of notations, denote the restriction of $s_{\sigma^L}$ to $V_{\sigma^L}$ also by $s_{\sigma^L}$. Then by \cite{FO99}, there is a map $\psi_{\sigma^L}$ mapping $s_{\sigma^L}^{-1}(0)/\Gamma_{\sigma^L}$ onto a neighborhood of $[\sigma^L]$ in $\mathcal{M}^L_p$. This finishes the review of the construction of a Kuranishi neighborhood $(V_{\sigma^L},E_{\sigma^L},\Gamma_{\sigma^L},s_{\sigma^L},\psi_{\sigma^L})$ of $[\sigma^L]\in\mathcal{M}^L_p$.

For a point in $\mathcal{M}^X_p$ represented by $\sigma^X=((\Sigma,z),\varphi)$, using Proposition \ref{prop4.4}, we decompose $\Sigma$ into irreducible components $\Sigma=\Sigma_0\cup\Sigma_1\cup\ldots\cup\Sigma_k$, where $\Sigma_0,\Sigma_1,\ldots,\Sigma_k$ are all copies of $\bP^1$, such that the restrictions of $\varphi$ to $\Sigma_0$ and $\bigcup_{a=1}^k\Sigma_a$ represent the classes $h$ and $\alpha$ respectively. The construction of a Kuranishi neighborhood $(V_{\sigma^X},E_{\sigma^X},\Gamma_{\sigma^X},s_{\sigma^X},\psi_{\sigma^X})$ of $[\sigma^X]\in\mathcal{M}^X_p$ is more or less the same as above, except that $W^{1,p}(\Sigma_0;\varphi^*(TX);L)$ is replaced by the space $W^{1,p}(\Sigma_0;\varphi^*(TX))$ of sections $v$ of $\varphi^*(TX)$ of class $W^{1,p}$.

We can now go back to the proof of the proposition.

Let $[\sigma^L]\in\mathcal{M}^L_p$, $[\sigma^X]\in\mathcal{M}^X_p$ be such that $j([\sigma^L])=[\sigma^X]$. First of all, it is obvious that the automorphism groups $\Gamma_{\sigma^L}$ and $\Gamma_{\sigma^X}$ are the same. Next, since the moduli space of maps from $(D^2,\partial D^2)$ to $(X,L)$ with class $\beta_0$ is unobstructed, we can choose $E_0=0$ for the obstruction space $E_{\sigma^L}$. Similarly, since the moduli space of maps from $\bP^1$ to $X$ with class $h$ is unobstructed, we can also choose $E_0=0$ for the obstruction space $E_{\sigma^X}$. Hence, the obstruction spaces $E_{\sigma^L}$ and $E_{\sigma^X}$ are both of the form $0\oplus E_1\oplus\ldots\oplus E_k$ and can be identified naturally.

We can identify $V_{\textrm{deform},\sigma^L}$ with $V_{\textrm{deform},\sigma^X}$ since the component $\Sigma_0$ in $\Sigma^L$ has no nontrivial deformations and the component $\Sigma_0$ in $\Sigma$ is unstable. It is also clear that we can identify $V_{\textrm{resolve},\sigma^L}$ with $ V_{\textrm{resolve},\sigma^X}$. Now, let $(u=(u_0,u_1,\ldots,u_k),\zeta)\in V_{\sigma^L}$. Because $E_0=0$, we have $D_\varphi\bar\partial u_0=0$. From the construction of the family of smooth maps $\varphi_{u,\zeta}:(\Sigma^L_\zeta,\partial\Sigma^L_\zeta)\to(X,L)$, it follows that the restriction of $\varphi_{u,\zeta}$ to the component $\Sigma_0$ is a holomorphic map with class $\beta_0$. We also have $\varphi_{u,\zeta}(z)=p$. But there is a unique (up to automorphisms of the domain) holomorphic map from $(D^2,\partial D^2)$ to $(X,L)$ with class $\beta_0$ whose boundary passes through $p$, which is given by $\varphi|_{\Sigma_0}$. So we must have $u_0=0$. By a similar argument, all $(u,\zeta)\in V_{\sigma^X}$ also have $u_0=0$. Therefore, we can naturally identify $V_{\sigma^L}$ and $V_{\sigma^X}$.

Finally, we can identify the families of maps $\{\varphi_{u,\zeta}:(\Sigma^L_\zeta,\partial\Sigma^L_\zeta)\to(X,L):(u,\zeta)\in V_{\sigma^L}\}$ with $\{\varphi_{u,\zeta}:\Sigma_\zeta\to X:(u,\zeta)\in V_{\sigma^X}\}$ by the gluing construction that we used in the definition of the map $j$. Hence, the maps $s_{\sigma^L}$ and $\psi_{\sigma^L}$ can also be naturally identified with the maps $s_{\sigma^X}$ and $\psi_{\sigma^X}$ respectively.

This completes the proof of the proposition.
\end{proof}
Theorem \ref{main_thm} now follows from Propositions \ref{prop4.1} and \ref{prop4.5}.

\section{Applications to mirror symmetry}\label{sec5}

In this section, we apply Theorem \ref{main_thm} to study mirror symmetry for the toric manifolds $X=\bP(K_Y\oplus\mathcal{O}_Y)$. We shall first briefly review the constructions of the mirrors for toric manifolds, following Cho-Oh \cite{CO03}, Auroux \cite{A07, A09}, Fukaya-Oh-Ohta-Ono \cite{FOOO08a, FOOO08b} and Chan-Leung \cite{CL08a, CL08b}.

As usual, $N\cong\bZ^n$ is a rank $n$ lattice, $M=\textrm{Hom}(N,\bZ)$ is the dual lattice and $\langle\cdot,\cdot\rangle:M\times N\to\bZ$ is the dual pairing. Also let $N_\bR=N\otimes_\bZ\bR$, $M_\bR=M\otimes_\bZ\bR$, and denote by $T_N$ and $T_M$ the real tori $N_\bR/N$ and $M_\bR/M$ respectively.

Let $X=X_\Delta$ be an $n$-dimensional smooth projective toric variety defined by a fan $\Delta$ in $N_\bR$. Let
$v_1,\ldots,v_d$ be the primitive generators of the 1-dimensional cones in $\Delta$. We equip $X$ with a toric K\"{a}hler structure $\omega$. Let $P$ be the corresponding moment polytope and $\mu:X\to P$ be the moment map. $P$ is defined by a set of inequalities
$$P=\{x\in M_\bR|\langle x,v_i\rangle\geq\lambda_i\textrm{ for }i=1,\ldots,d\},$$
for some $\lambda_1,\ldots,\lambda_d\in\bR$. For $i=1,\ldots,d$, we let $l_i:M_\bR\rightarrow\bR$ be the affine linear function defined by $l_i(x)=\langle x,v_i\rangle-\lambda_i$.

We are interested in the mirror symmetry for the K\"{a}hler manifold $X$, equipped with the toric K\"{a}hler structure $\omega$ and the nowhere zero meromorphic $n$-form $\Omega=d\log w_1\wedge\ldots\wedge d\log w_n$, where $w_1,\ldots,w_n$ are the standard complex coordinates on the open dense orbit $U=N\otimes_\bZ\bC^*\cong(\bC^*)^n\subset X$. From the point of view of Auroux \cite{A07}, we are looking at the mirror symmetry for $X$ relative to the toric divisor $D_\infty=\bigcup_{i=1}^d D_i=X\setminus U$. As before, $D_i$ is the toric prime divisor in $X$ corresponding to $v_i$. The mirror geometry is given by a Landau-Ginzburg model $(X^\vee,W)$ consisting of a bounded domain $X^\vee\subset(\bC^*)^n$ and a holomorphic function $W:X^\vee\rightarrow\bC$ called the mirror superpotential.

As discussed in Auroux \cite{A07} and Chan-Leung \cite{CL08a, CL08b}, the mirror manifold $X^\vee$ can be obtained by dualizing Lagrangian torus fibrations (so-called T-duality) as follows. Restricting the moment map $\mu:X\to P$ to the open dense orbit $U\subset X$ gives a torus bundle $\mu:U\to\textrm{Int}(P)$, where $\textrm{Int}(P)$ denotes the interior of the polytope $P$. In fact this bundle is trivial, so we have $U=\textrm{Int}(P)\times\sqrt{-1}T_N$. The mirror manifold $X^\vee$ is given by the total space of the dual torus bundle, i.e.
$$X^\vee=\textrm{Int}(P)\times\sqrt{-1}T_N^\vee=\textrm{Int}(P)\times\sqrt{-1}T_M.$$
$X^\vee$ comes with a natural K\"{a}hler structure. In particular, as a complex manifold, $X^\vee$ is biholomorphic to a bounded domain in $(\bC^*)^n=M_\bR\times\sqrt{-1}T_M$. If $y=(y_1,\ldots,y_n)\in(\bR/2\pi\bZ)^n$ are the fiber coordinates on $T_M$ and the complex coordinates on $(\bC^*)^n$ are given by $z_j=\exp(-x_j-\sqrt{-1}y_j),\ j=1,\ldots,n$, where $x=(x_1,\ldots,x_n)\in\textrm{Int}(P)$, then $X^\vee\subset(\bC^*)^n$ can be written as
$$X^\vee=\{(z_1,\ldots,z_n)\in(\bC^*)^n:|e^{\lambda_i}z_i|<1,\ i=1,\ldots,d\}.$$

Geometrically, $X^\vee$ should be viewed as the moduli space of pairs $(L,\nabla)$ consisting of a (special) Lagrangian torus fiber of the moment map $\mu:X\rightarrow P$ together with a flat $U(1)$-connection $\nabla$ on the trivial line bundle $\underline{\bC}$ over $L$. More precisely, to a point $z=(z_1=\exp(x_1+\sqrt{-1}y_1),\ldots,z_1=\exp(x_1+\sqrt{-1}y_1))\in X^\vee$, we associate the flat $U(1)$-connection $\nabla_y=d+\frac{\sqrt{-1}}{2}\sum_{j=1}^n y_j du_j$ on the trivial line bundle $\underline{\bC}$ over the Lagrangian torus $L_x=\mu^{-1}(x)\cong T_N$, where $u=(u_1,\ldots,u_n)\in(\bR/2\pi\bZ)^n$ are the fiber coordinates on $T_N$. This picture is motivated by the SYZ conjecture for mirror Calabi-Yau manifolds proposed by Strominger, Yau and Zaslow \cite{SYZ96} in 1996.

On the other hand, it turns out that the mirror superpotential $W:X^\vee\to\bC$ acts as the mirror of the obstruction $\mathfrak{m}_0$ to the Floer homology of Lagrangian torus fibers in $X$.\footnote{In their works \cite{FOOO06, FOOO08a, FOOO08b}, Fukaya, Oh, Ohta and Ono call $W$ the potential function and they define it over the Novikov ring $\Lambda_0$ instead of $\bC$.} As shown in \cite{FOOO06}, $\mathfrak{m}_0$ comes from the virtual counting of Maslov index two holomorphic disks in $X$ with boundary in the Lagrangian torus fibers $L$. This leads to the following expression for $W$: For $\beta\in\pi_2(X,L)$, we define a holomorphic function $Z_\beta$ on $X^\vee$ by
$$Z_\beta(L,\nabla)=\exp\Big(-\frac{1}{2\pi}\int_\beta\omega\Big)\textrm{hol}_\nabla(\partial\beta).$$
Then the mirror superpotential $W:X^\vee\rightarrow\bC$ is given by the following holomorphic function
\begin{equation}\label{superpotential}
W(L,\nabla)=\sum_{\beta\in\pi_2(X,L),\ \mu(\beta)=2}c_\beta Z_\beta(L,\nabla),
\end{equation}
assuming that the sum converges. See Cho-Oh \cite{CO03}, Auroux \cite{A07, A09} and Fukaya-Oh-Ohta-Ono \cite{FOOO08a, FOOO08b} for more details.

For $i=1,\ldots,d$, let $\beta_i\in\pi_2(X,L)$ be the relative homotopy class such that $D_j\cdot\beta_i=\delta_{ij}$. Then, by the symplectic area formula of Cho and Oh (Theorem 8.1 in \cite{CO03}), we have
$$\int_{\beta_i}\omega=2\pi l_i(x)=2\pi(\langle x,v_i\rangle-\lambda_i),$$
where $x\in\textrm{Int}(P)$ is the image of $L$ under the moment map (i.e. $L=\mu^{-1}(x)$). Hence, for the basic classes $\beta_i$, the function $Z_{\beta_i}$ is given in local coordinates by
$$Z_{\beta_i}(L_x,\nabla_y)=\exp(-l_i(x))\exp(-\sqrt{-1}\langle y,v_i\rangle)=e^{\lambda_i}z^{v_i},$$
where $z^v$ denotes the monomial $z_1^{v^1}\cdots z_n^{v^n}$.

Furthermore, by Theorem \ref{Cho-Oh}, we have $c_{\beta_i}=1$ for $i=1,\ldots,d$. In particular, when $X$ is Fano (i.e. the anticanonical bundle $K_X^{-1}$ is ample), $\beta_1,\ldots,\beta_d\in\pi_2(X,L)$ are the only Maslov index two classes. Hence, the mirror superpotential is given explicitly by $$W=Z_{\beta_1}+\ldots+Z_{\beta_d}=e^{\lambda_1}z^{v_1}+\ldots+e^{\lambda_d}z^{v_d}.$$
However, in the non-Fano cases, the invariants $c_\beta$ and hence $W$ are in general very hard to compute. The only non-Fano examples whose mirror superpotentials are explicitly computed are the Hirzebruch surfaces $F_2$ and $F_3$, first computed by Auroux in \cite{A09}. Later, Fukaya, Oh, Ohta and Ono gave a different proof for the $F_2$ case in \cite{FOOO10a}.

Let's go back to our toric manifolds $X=\bP(K_Y\oplus\mathcal{O}_Y)$. We want to compute their mirror superpotentials using Theorem \ref{main_thm}.
\begin{lem}\label{lem5.1}
If $\beta\in\pi_2(X,L)$ is a Maslov index two class with $c_\beta\neq0$, then $\beta$ must either be one of $\beta_1,\ldots,\beta_m,\beta_{m+1}$ or of the form $\beta_0+\alpha$ for some effective class $\alpha\in H_2(X,\bZ)$ with $c_1(\alpha)=0$.
\end{lem}
\begin{proof}
First of all, since $X$ is semi-Fano, $c_1(\alpha)\geq0$ for any effective class $\alpha\in H_2(X,\bZ)$. Hence, if $\beta\in\pi_2(X,L)$ is a Maslov index two class, then it must be of the form $\beta_i+\alpha$ for some $i=0,1,\ldots,m,m+1$ and some effective class $\alpha\in H_2(X,\bZ)$ with $c_1(\alpha)=0$. Let $\varphi:(\Sigma^L,\partial\Sigma^L)\to(X,L)$ be a stable map from a genus 0 bordered Riemann surface $(\Sigma^L,\partial\Sigma^L)$ to $(X,L)$ representing the class $\beta_i+\alpha$. Suppose that $\alpha\neq0$. Then, by the proof of Proposition \ref{prop4.2}, we can decompose $\Sigma^L$ into $\Sigma^L_0\cup\Sigma_1$, where $\Sigma^L_0=D^2$ is a disk and $\Sigma_1$ is a genus 0 nodal curve, such that the restrictions $\varphi_0:=\varphi|_{\Sigma^L_0}$ and $\varphi_1:=\varphi|_{\Sigma_1}$ represent the classes $\beta_i$ and $\alpha$ respectively. However, by Lemma \ref{lem3.2}(2), the image of $\varphi_1$ must be contained entirely in the toric prime divisor $D_0$. Since $\varphi(\Sigma^L_0)\cdot D_0=\delta_{0i}$ and the domain of $\varphi$ is connected, we must have $i=0$. Hence $c_{\beta_i+\alpha}=0$ unless $i=0$ or $\alpha=0$.
\end{proof}
\begin{thm}\label{thm5.1}
For the $\bP^1$-bundle $X=\bP(K_Y\oplus\mathcal{O}_Y)$ over a toric Fano manifold $Y$, the mirror superpotential $W:X^\vee\rightarrow\bC$ is given by
$$W=CZ_{\beta_0}+Z_{\beta_1}+\ldots+Z_{\beta_m}+Z_{\beta_{m+1}},$$
where
$$C=1+\sum_{\substack{\alpha\in H_2^\textrm{eff}(X,\bZ),\\ \alpha\neq0,\ c_1(\alpha)=0}} \textrm{GW}^{X,h+\alpha}_{0,1}(\textrm{[pt]})q^\alpha,$$
and $q^\alpha=\exp(-\frac{1}{2\pi}\int_\alpha\omega)$.
\end{thm}
\begin{proof}
This is a consequence of formula (\ref{superpotential}), Lemma \ref{lem5.1} and Theorem \ref{main_thm}.
\end{proof}

\noindent\textbf{Example: The Hirzebruch surface $\mathbb{F}_2$.} Consider $X=\mathbb{F}_2=\bP(K_{\bP^1}\oplus\mathcal{O}_{\bP^1})$. We choose the primitive generators of the 1-dimensional cones in the fan $\Delta$ defining $\mathbb{F}_2$ to be\footnote{Note that this choice of generators is different from the one in Section \ref{sec3}. This does not alter any of our results. We make this choice just to make our notations consistent with those in Auroux \cite{A09} and Fukaya-Oh-Ohta-Ono \cite{FOOO10a}.}
$$v_0=(0,-1),v_1=(1,0),v_2=(-1,-2),v_3=(0,1)$$
in $N=\bZ^2$. We equip $\mathbb{F}_2$ with a toric K\"{a}hler structure so that moment polytope $P$ is given by
$$P=\{(x_1,x_2)\in\bR^2|x_1\geq0, x_2\geq0, x_2\leq t_2, x_1+2x_2\leq t_1+2t_2\},$$
where $t_1,t_2>0$. See Figure \ref{fig1} below.

\begin{figure}[ht]
\setlength{\unitlength}{1mm}
\begin{picture}(100,35)
\put(15,17){\vector(1,0){15}} \put(31,16){$v_1$} \put(15,17){\vector(0,1){15}} \put(14,33){$v_3$} \put(15,17){\vector(0,-1){15}} \put(14,-0.5){$v_0$} \put(15,17){\vector(-1,-2){8}} \put(5,-1.5){$v_2$} \put(50,5){\vector(0,1){25}} \put(49,31){$x_2$} \put(50,5){\vector(1,0){48}} \put(99,4.5){$x_1$} \curve(50,20, 60,20) \curve(60,20, 90,5) \curve(50.8,20, 49.2,20) \curve(90,4.2, 90,5.8) \put(48,3){0} \put(46,19){$t_2$} \put(85.5,1){$t_1+2t_2$} \put(45.5,11.5){$D_1$} \put(65,1.5){$D_3$} \put(52,21){$D_0$} \put(73,14){$D_2$}
\end{picture}
\caption{The fan $\Delta$ defining $\mathbb{F}_2$ (left) and its moment polytope $P$ (right).}\label{fig1}
\end{figure}
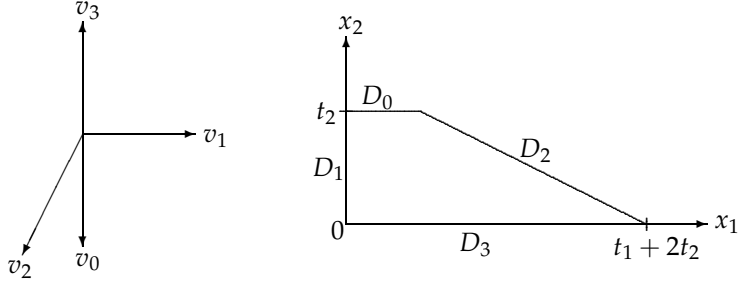

The effective cone $H_2^\textrm{eff}(\mathbb{F}_2,\bZ)$ is generated by two primitive relations
$$v_0+v_3=0,\ v_1+v_2-2v_0=0.$$
Let $h:=(1,0,0,1),\alpha:=(-2,1,1,0)\in H_2^\textrm{eff}(\mathbb{F}_2,\bZ)$ be the corresponding homology classes, which represent the fiber and the base of $\mathbb{F}_2$ respectively. Then
$$t_1=\int_\alpha\omega_X,\ t_2=\int_h\omega_X.$$
Let $q_i=\exp(-t_i)$ for $i=1,2$. We also have $c_1(h)=2$ and $c_1(\alpha)=0$.

Now, the mirror manifold $X^\vee$ is a bounded domain in $(\bC^*)^2$. By Theorem \ref{thm5.1}, the mirror superpotential $W:X^\vee\rightarrow\bC$ is given by
$$W=CZ_{\beta_0}+Z_{\beta_1}+Z_{\beta_2}+Z_{\beta_3}=C\frac{q_2}{z_2}+z_1+\frac{q_1q_2^2}{z_1z_2^2}+z_2,$$
where
$$C=\sum_{k=0}^\infty\textrm{GW}^{F_2,h+k\alpha}_{0,1}(\textrm{PD[pt]})q_1^k,$$
and $z_1,z_2$ are the standard coordinates on $(\bC^*)^2$. $\mathbb{F}_2$ is symplectomorphic to $\mathbb{F}_0=\bP^1\times\bP^1$ with induced isomorphism on degree-2 homology given by
\begin{eqnarray*}
H_2(\mathbb{F}_2,\bZ) & \rightarrow & H_2(\mathbb{F}_0,\bZ)\\
\alpha & \mapsto & l_1-l_2\\
h & \mapsto & l_2,
\end{eqnarray*}
where $l_1,l_2\in H_2(\mathbb{F}_0,\bZ)$ are the line classes in the two $\bP^1$ factors. Since Gromov-Witten invariants are symplectic invariants, the Gromov-Witten invariants of $\mathbb{F}_2$ are all equal to those of $\mathbb{F}_0$. So we have
\begin{eqnarray*}
\textrm{GW}^{\mathbb{F}_2,h+k\alpha}_{0,1}(\textrm{PD[pt]}) & = & \textrm{GW}^{\mathbb{F}_0,kl_1+(1-k)l_2}_{0,1}(\textrm{PD[pt]})\\
& = & \left\{\begin{array}{ll}
                1 & \textrm{if $k=0$ or $k=1$}\\
                0 & \textrm{otherwise.}
             \end{array}\right.
\end{eqnarray*}
Hence, $c_{\beta_0+k\alpha}=0$ for $k\geq2$ and $c_{\beta_0+\alpha}=c_{\beta_0}=1$. We conclude that $C=1+q_1$ and the mirror superpotential is given by
$$W=z_1+z_2+\frac{q_1q_2^2}{z_1z_2^2}+\frac{q_2+q_1q_2}{z_2}.$$
This agrees with the formula in Proposition 3.1 in Auroux \cite{A07}.\hfill $\square$\\

The formula in Theorem \ref{main_thm} has been applied to investigate mirror symmetry for various classes of toric manifolds. In \cite{LLW10a}, the formula was generalized and used to compute open Gromov-Witten invariants for toric Calabi-Yau 3-folds. In \cite{CLL10} and \cite{LLW10b}, the formula and its generalization in \cite{LLW10a} were used to obtain an enumerative meaning for the (inverse) mirror maps for toric Calabi-Yau 2- and 3-folds. In particular, this explains why we always get integral coefficients for the Taylor expansions of these mirror maps. In \cite{CL10}, the formula was used to compute mirror superpotentials for all semi-Fano toric surfaces.

\end{document}